\documentclass[12pt]{amsart}
\RequirePackage{lmodern} 

\usepackage{amsfonts}
\usepackage{amsmath}
\usepackage{amssymb}
\usepackage{amsthm}
\usepackage{appendix}
\usepackage{changepage}
\usepackage{enumerate}
\usepackage{enumitem}
\usepackage{eucal}
\usepackage{latexsym}
\usepackage{mathrsfs}
\usepackage{scrextend}
\usepackage{textcomp}
\usepackage{todonotes}
\usepackage{verbatim}
\usepackage{pdflscape}
\usepackage{lscape}

\relax							

\newlength{\pardefault}							
\newcommand{\innerProductReg}[2]{\left\langle #1 , #2 \right\rangle}
\newcommand{\innerProductTri}[3]{\left\langle #1 , #2 \right\rangle_{#3}}

\newcommand{\norm}[1]{\| #1  \|}

\newcommand{\normTre}[3]{\| #1  \| _{#2}^{#3}}

\newcommand\restr[2]{{
  \left.\kern-\nulldelimiterspace 
  #1 
  \vphantom{\big|} 
  \right|_{#2} 
  }}
\newcommand{\defEq}{\overset{\text{def} } {=}}

\newcommand{\forceOddPage}{
\ifodd\theCurrentPage\pagebreak\fi}

\newcommand{\forceEvenPage}{
\ifodd\value{\theCurrentPage}\newpage\else\pagebreak\fi}

\newcommand*\cleartoleftpage{%
  \clearpage
  \ifodd\value{page}\hbox{}\newpage\fi
}

\newcommand*\cleartorightpage{%
  \clearpage
  \unless\ifodd\value{page}\hbox{}\newpage\fi
}

\newcommand{\ignore}[1]{}

\newtheorem{Pa}{Paper}[section]

\newtheorem{Rk}[Pa]{{\bf Remark}}

\newtheorem{Ex}[Pa]{{\bf Example}}
\newtheorem{theorem}[Pa]{{\bf Theorem}}
\newtheorem{lem}[Pa]{{\bf Lemma}}

\newtheorem{definition}[Pa]{{\bf Definition}}
\newtheorem{proposition}[Pa]{{\bf Proposition}}

\newenvironment{pf}[1][\unskip]{
\par
\noindent
\paragraph{{\bf Proof #1:}}
\noindent
}
{\hfill$\square$\\}

\def\N{\mathbb N}

\usepackage[hypertexnames=false]{hyperref}

\hypersetup{
  colorlinks,
  citecolor	= violet,
  linkcolor	= blue,
  urlcolor	= red}

% ...

\begin{document}

\title[Stochastic Wiener filter in the white noise space]
{Stochastic Wiener filter in the white noise space}

\author[D. Alpay]{Daniel Alpay}
\address{(DA) Schmid College of Science and Technology,
Chapman University, One University Drive Orange, California 92866, USA}
\email{alpay@chapman.edu}

\author[A. Pinhas]{Ariel Pinhas}
\address{(AP) Department of mathematics,
Ben-Gurion University of the Negev, P.O. Box
653, Beer-Sheva 84105, Israel}
\email{arielp@post.bgu.ac.il}

\date{}

\thanks{The authors thank the Foster G. and Mary McGaw Professorship in 
Mathematical Sciences, which supported this research}

\begin{abstract}
In this paper we introduce a new approach to the study of filtering
theory by allowing the system's parameters to have a random character. 
We use Hida's white noise space theory to
give an alternative characterization and a proper generalization to the Wiener filter
over a suitable space of stochastic distributions introduced by Kondratiev.
The main idea throughout this paper is to use the nuclearity of this spaces in order to view the random variables as 
bounded multiplication operators (with respect to the Wick product) between Hilbert spaces of stochastic distributions.
This allows us to use operator theory tools 
and properties of Wiener algebras over Banach spaces to proceed and
characterize the Wiener filter
equations under the underlying randomness assumptions.
\end{abstract}


\maketitle
\setcounter{tocdepth}{1}

\tableofcontents

\section{Introduction}
\label{SectionInto}
\setcounter{equation}{0}

The fundamental question in filtering theory is how to recover a signal from incomplete or distorted information. 
The Wiener filter, under stationarity assumptions, is the best linear minimum mean square error filter. 
In order to lay the foundation for the Wiener filter (see \cite{wiener1949extrapolation}), 
one needs to know the spectral structure of the noise and of the signal in advance. 
In the non-causal case, the Wiener filter is given in terms of the two sided 
$z$-transform 
\[
H(z) = \frac{S_{uy}(z)}{S_{y}(z)},
\]
where $S_{y}(z)$ and $S_{uy}(z)$ are the $z$-transforms of the autocorrelation function of the output process 
and the crosscorrelation function of the output and the input signals, respectively.
The causal Wiener filter is based on factorization theorems over Wiener algebras and is expressed by
\[
H(z) = \left[ S_{uy}(z)S^-_{y}(z)^{-1} \right]_{+} S^+_{y}(z)^{-1},
\]
where the operator $[S]_+$ denotes the causal part of an element $S$.
The theory is well-developed and complete as long as the spectral structure is known and non-random. 
However, when uncertainty is allowed in the spectral structure, the relation is not clear. 
Some solutions to this question may be found in \cite{kassam1977robust,kassam1985robust,poor1980robust,vastola1984robust}.
The main approaches to handle such uncertainty is to examine the average spectral behavior or, alternatively, 
to determine thresholds such that the estimated signal remains bounded in some sense.
\smallskip

These two approaches have some disadvantages. 
The first, considering the mean behavior of the spectral structure, misses some of the stochastic data and, 
furthermore, some of the results can cause divergence in data recovery. 
The former, detecting thresholds to ensure bounded result tends to be too restricted and conceals some of the stochastic behavior of the spectral structure.
\smallskip

At this stage, it is worth mentioning a different approach where positive definite functions $R(n)$ 
on the integers and with values continuous operators from a Banach space into its anti-dual play an important
role in the theory of stochastic processes and related topics (\cite{MR0500218,MR3429638,MR0420807,MR0350478,MR665179} and \cite{atv1}). 
An important feature of Banach spaces (not always used in the above references) is that
they possess the factorization property: a positive operator from a Banach space into its anti-dual can be factorized via a Hilbert space; see \cite{MR647140} and e.g. \cite{atv1} 
for a proof.\smallskip

In \cite{MR2610579} a study of linear stochastic systems within the framework of the white noise
space has been initiated. The usual input-output relation was replaced by
\begin{equation}
\label{eq123}
y_n = \sum_{m \in \mathbb Z} h_{n-m} \circ u_m, \quad n\in \mathbb Z,
\end{equation}
where $y_n, u_n$ and $h_n$ are random variables in the space of stochastic distribution $\mathscr {S}_{-1}$ introduced by Kondratiev (see \cite{MR1408433,MR1407325}), 
and $\circ$ denotes the Wick product. 
This approach was continued in \cite{alp} with the study of the state-space equations within the space of stochastic distributions.\\

In this paper we study the Wiener filter when the system parameters are elements in the space of stochastic distributions $\mathscr {S}_{-1}$ rather than deterministic numbers. 
In this way, the randomness is inherited naturally in the system.
We replace the point-wise product by the Wick product. Using the nuclearity, we show that 
the autocorrelation function becomes a sequence of bounded operators in a given Hilbert space $H$, 
while its spectral density (namely, the $z$-domain) belongs to $\mathcal W_{\mathcal B(H)}$, the Wiener algebra over the Banach space of bounded operators over Hilbert space.
Some of the fundamental factorization theorems in $\mathcal W_{\mathcal B(H)}$ are utilized in order to characterize the Wiener filter in the space of stochastic distributions.  
As a consequence, the randomness assumption reflects to uncertainty under the spectral mapping.
\smallskip

In classical linear system theory, $z$-transforms of the underlying sequences are usually assumed to converge, and play an important role in the arguments. For instance, the
transfer function is the $z$-transform of the impulse response $h_0,h_1,\ldots$. In the present setting it is important to note that
if $(x_n)$ is a second-order stationary process in the probability space $\mathbf L_2(\Omega,\mathcal B,P)$, the power series $\sum_{n=0}^\infty z^nx_n$ 
converges in $\mathbf L_2(\Omega,\mathcal B,P)$. In the present approach, we replace these two convergences by convergence in $\mathscr {S}_{-1}$ 
(see Subsection \ref{rk1234}).
\smallskip

Finally, we mention that the suggested approach is a proper generalization of the conventional Wiener filter.
As soon as one assumes that the system is described by a deterministic impulse response sequence (that is, the $h_n$ in \eqref{eq123} are complex numbers and not stochastic
distributions) and input and output are now discrete signals with values in the $L^2$ white noise space,
the Wick product becomes the scalar multiplication and we consider the Wiener algebra with scalar coefficients.
See Remark \ref{rkdetermCase} below for more details.
\smallskip

We now turn to the outline of the paper.
To develop filtering theory over Hida's white
noise space, we first give in Section \ref{SectionWNS} a short, but necessary, survey 
of the foundation of the white noise space theory and of the relevant tools.
The Wiener algebra over a Banach space and an associated factorization theorem
is presented in Section \ref{SectionKondratievSpcae}.
Section \ref{SectionStochProcess} is dedicated to the construction of a framework of stochastic processes 
over the space of stochastic distributions.
The non-causal Wiener filter within the white noise space is studied in Section \ref{SectionWienerFilter}.
Finally, in Section \ref{SectionCausalWienerFilter}, we develop the causal Wiener filter in the white noise space setting.

\begin{Rk}
We restrict this paper to the case of discrete time models. 
Similar results may be obtained for continuous models.
\end{Rk}
\section{The white noise space and the Kondratiev space}
\label{SectionWNS}
\setcounter{equation}{0}

\subsection{The white noise space}
To begin with, consider the
Schwartz space $\mathcal S$ of {\sl real-valued} smooth functions
which, together with their derivatives, decrease rapidly to zero
at infinity. For $s\in{\mathcal S}$, let $\|s\|$ denote its
${\mathbf L}_2({\mathbb R})$ norm. The function
\[
K(s_1-s_2)=e^{-\frac{\|s_1-s_2\|^2}{2}}
\]
is positive (in the sense of reproducing kernels) for $s_1,s_2$
belonging to ${\mathcal S}$. 
Since ${\mathcal S}$ is nuclear, we may use an extension
of Bochner's Theorem for nuclear spaces, due to Minlos (see
\cite{MR0154317}, \cite[Th\'eor\`eme 3, p. 311]{MR35:7123}):
there exists a probability measure $P$ on ${\mathcal S}^\prime$
such that
\[
K(s)=\int_{{\mathcal S}^\prime}e^{-\imath\langle
s^\prime,s\rangle}dP(s^\prime),
\]
where $\langle s^\prime,s\rangle$ denotes the duality
between ${\mathcal S}$ and ${\mathcal S}^\prime$. 
The triple $\mathcal W \defEq ({\mathcal S}^\prime, {\mathcal F},dP)$, where
${\mathcal F}$ is the Borel $\sigma$-algebra, is called the
white noise space. 
For a given $s \in \mathcal S$, the formula 
$Q_s(w) = \innerProductReg{w}{s}$ is a centered Gaussian variable 
with covariance $\normTre{s}{}{2}$.
\smallskip

A certain orthogonal basis of the real space 
${\mathbf L}_2({\mathcal S}^\prime, {\mathcal F},dP)$
plays a special role and is constructed in terms of Hermite
functions.
Its elements are denoted by $H_\alpha$ and given by
\[
H_{\alpha} = \prod_{j} h_{\alpha_j}(Q_{\xi_j})
\]
where $h_{k}$ is the $k$-th Hermite polynomial
and where $\xi_{j}$ is the $j$-th normalized Hermite function.
Moreover, the
index $\alpha$ runs through the set $\ell$ of sequences
$(\alpha_1,\alpha_2,\ldots)$, whose entries are in
\[
\mathbb N_0=\left\{
0,1,2,3,\ldots\right\},
\]
and $\alpha_k\not =0$ for all but a finite number of indices $k$ and
(see \cite{MR1408433}). With the multi-index notation
\[
\alpha!=\alpha_1!\alpha_2!\cdots,
\]
we have
\begin{equation}
\label{eq:fock1}
\|H_\alpha\|_{\mathcal W}^2=\alpha!.
\end{equation}
The Wick product in ${\mathcal W}$ is defined by
\[
H_\alpha \circ H_\beta
=
H_{\alpha+\beta},
\quad 
\alpha,\beta\in\ell.
\]


\subsection{The Kondratiev Space}
\label{SectionKondratievSpcae}
The space $\mathbf L_2(\mathcal W)$ is not stable under the Wick
product. This motivates the definition of the Kondratiev space, denoted by $\mathscr {S}_{-1}$, which contains  $\mathbf L_2(\mathcal W)$. 
We first introduce a family of Hilbert spaces $\mathcal{H}_k$
containing $\mathbf L_2(\mathcal W)$. Let $k \in \mathbb{N}$ and let $\mathcal{H}_k$
be the collection of formal series
\begin{equation}
\label{Hk}
f(\omega)=\sum_{\alpha \in \ell}{c_\alpha H_\alpha(\omega)},
\quad c_\alpha \in \mathbb{C},
    \end{equation}
such that,
\begin{equation}
\label{norm_S-1}
||f||_k
\stackrel{\rm def.}{=}
(\sum_{\alpha \in \ell}{|c_\alpha|^2
(2\mathbb{N})^{-k \alpha}})^{1/2}<\infty,
\end{equation}
    where
    \[
    (2{\mathbb N})^\alpha=(2\times 1)^{\alpha_1}(2\times 2)^{\alpha_2} (2\times
3)^{\alpha_3}\cdots.
    \]
The product makes sense since only a finite number of $\alpha _i \neq 0$.
We note that
\[
\mathcal W \subseteq \mathcal{H}_1 \subseteq
\mathcal{H}_2\subseteq \cdots \subseteq \mathcal{H}_k \subseteq
\cdots.
\]
The Kondratiev space $\mathscr {S}_{-1}$ is the inductive limit of
the spaces $\mathcal{H}_k$,
\[
\mathscr {S}_{-1}\overset{ \text{def} }{=}\bigcup_{k \in \mathbb{N}}{\mathcal{H}_k},
\]
and is stable under the Wick product, see below.
The Kondratiev space of stochastic test functions, $\mathscr {S}_1$, is the
space of all functions of the form
\[
f(\omega)=\sum_{\alpha \in \mathcal{J}}{c_\alpha
H_\alpha(\omega)},\quad c_\alpha\in\mathbb C,
\]
such that $\sum _{\alpha \in \mathcal{J}} { (2\mathbb N )^{k
\alpha }(\alpha ! )^2 |c_\alpha|^2 } < \infty$ for all $k \in
\mathbb N _0$. It is a countably normed space, contained in
$\mathbf L_2(\mathcal W)$. The spaces $\mathscr {S}_{-1}$ and $\mathscr {S}_1$ are nuclear spaces,
dual of each other, and together with $\mathcal W$, form a
Gelfand triple,
\[
\mathscr {S}_1 \subseteq \mathcal W \subseteq \mathscr {S}_{-1}.
\]
The V\r{a}ge's inequality, presented below, plays a major role and shows that $\mathscr {S}_{-1}$ is stable under the Wick product. 
It allows us to consider the multiplication operator 
(with the appropriate assumptions) as bounded operators.
\begin{theorem}\cite[V\r{a}ge's inequality. p. 118]{MR1408433}
\label{vage1} Fix $k,\ell \in \mathbb{N}$ with $k>\ell+1$. Let $h
\in \mathcal{H}_\ell $ and let $f \in \mathcal{H}_k$. Then,
    \[
    ||h \circ f ||_k \leq A(k - \ell)||h ||_\ell || f ||_k.
    \]
    where
    \[
    A(k-\ell)
    \overset{ \text{def} }{=}
     \left( \sum_{\alpha}{2 \mathbb{N} ^{(\ell - k )\alpha }} \right) ^{\frac{1}{2}}
      < \infty.
    \]
\end{theorem}
For a proof that $A(k-l)$ is finite, see
\cite[Proposition 2.3.3, p. 31]{MR1408433}.\smallskip

Finally, in order to present a well defined generalization to the nonrandom parameter  case, the observation below suggests that the Wick product is reduced to a pointwise multiplication when one of the multipliers is nonrandom.
\begin{lem}[\cite{MR1408433}]
\label{WickDet}
Let $F,G \in \mathscr {S}_{-1}$ with $G=g_0 \in \mathbb C$. Then
\[
F \circ G= F \cdot g_0.
\]
\end{lem}

We conclude with the following remark.

\begin{Rk}
\label{rkKond}
The Kondratiev space is not a metric space, yet we note that convergent sequences in $\mathscr {S}_{-1}$ have 
the following important property (see \cite[P. 58 Th\'{e}or\`{e}me 4]{GS2}, in a more general setting).
A sequence $\left( x_n \right)_{n \in \mathbb Z}$ converges to some $x$ in  $\mathscr {S}_{-1}$
if and only if
there exist $N,k>0$ such that for all $n>N$, $x_n$ and $x$ are belong to $\mathcal H_k$ and
\[
\norm{x_n - x}_k \rightarrow 0.
\]
\end{Rk}
This property is used in Section \ref{SectionStochProcess} to provide 
motivation to one of our hypothesis, namely Condition \ref{finiteNormSum}.\\


\section{Wiener algebras over Banach spaces}
\label{SectionWienerAlg}
\setcounter{equation}{0}

It is well known that the causal Wiener filter is related to 
factorizations over Wiener algebras.
While in the classical case, the Wiener algebra consists of the elements of the form
\[
w= \sum_{n=-\infty}^\infty{w_n}z^n, \qquad w_n \in \mathbb C ^{p \times p},
\]
such that $\sum_{n=-\infty}^{\infty} \norm{w_n} < \infty$,
in our case we replace $\mathbb C ^{p \times p}$ by $\mathcal B(H)$ and consider the 
Wiener algebra associated with bounded operators over a Hilbert space.
We denote this space by $\mathcal{W}(\mathcal{B})$. 

Gohberg and Leiterer studied in \cite{MR48:12113,MR48:12114} the Wiener algebras $\mathcal{W}(\mathcal{B})$ when $\mathcal B$ is the Banach algebra of linear bounded operators in a Hilbert space. 
In their papers, they prove spectral factorization theorems in $\mathcal{W}(\mathcal{B})$ 
(see also the unpublished manuscript \cite{alpay2003spectral} for a related result). 
In the present paper we use a different version of the
factorization theorem \cite[Lemma II.1.1]{MR998278}. 
Let $H$ be a complex Hilbert space and let $\mathcal {B} (H)$ be the Banach algebra of bounded linear
operators on $H$. 
We denote by $\mathcal{W}_{\mathcal {B} (H)}$ the Wiener algebra on the unit circle $\mathbb{T}$ with Fourier coefficients in $\mathcal {B} (H)$.
\begin{theorem}[\cite{alpay2003spectral} and {\cite[Lemma II.1.1]{MR998278}} Spectral  decomposition - Operator-valued version]
\label{Spectral} 
Let $H$ be a separable Hilbert space and let
$\mathcal B=\mathcal B(H)$ be the associated Banach algebra of bounded operators. Let $W \in \mathcal{W} (\mathcal{B})$ and assume that $W(e^{\imath t})$ is positive definite for all $t\in\mathbb R$. 
Then there exists $W_+ \in
\mathcal{W}_+ (\mathcal{B})$ such that $W_+^{-1} \in \mathcal{W}_+
(\mathcal{B})$  and $W=W_+ ^* W_+$.
\end{theorem}

\section{$\mathscr {S}_{-1}$ valued Stochastic processes} 
\label{SectionStochProcess}
\setcounter{equation}{0}

We define wide-sense stationary stochastic processes in the white
noise space setting. The main idea is to look at random variables
as multiplication operators between Hilbert spaces of stochastic
distributions.

\subsection{The classical case}
We first look at the equivalent definitions in the classical case of second-order
discrete-time stochastic processes. We express the autocorrelation
function in terms of multiplication operators. 
Consider a
probability space ${\mathbf L}_2(\Omega, {\mathcal A},P)$.
We associate to each element 
$y\in{\mathbf L}_2(\Omega, {\mathcal A},P)$, 
the multiplication operator
\[
M_y:\,\,\,{\mathbb C}\longrightarrow {\mathbf
L}_2(\Omega, {\mathcal A},P).
\]
defined by
\[
M_y(1)=y.
\]
In the next lemma, we denote by $E(y)$ the
expectation of the random variable $y\in{\mathbf
L}_1(\Omega,{\mathcal A}, P)$.

\begin{proposition}
Let $(y_n)_{n\in{\mathbb Z}}$ be a second-order
discrete-time stochastic process. Then
\begin{equation}
\label{extpectMult}
E(y_{n_1}y_{n_2}^*)=(M_{y_{n_2}})^*M_{y_{n_1}},
\quad n_1,n_2\in{\mathbb Z}.
\end{equation}
\end{proposition}
\begin{proof}[\textbf{Proof :}]
Indeed, by definition of the adjoint operator,
\[
\begin{split}
\langle
M_{y_{n_2}}^*M_{y_{n_1}}1,1\rangle_{\mathbb C}&=
\langle M_{y_{n_1}}1,
M_{y_{n_2}}1\rangle_{{\mathbf
L}_2(\Omega, {\mathcal A},P)}\\
&=\langle y_{n_1},y_{n_2}\rangle_{{\mathbf
L}_2(\Omega, {\mathcal
A},P)}=E(y_{n_1}y_{n_2}^*).
\end{split}
\]
\end{proof}

To justify the assumption given in Definition \ref{stat} below, we first consider the classical setting.
Let $\left( x_n \right)_{n \in \mathbb Z}$ be a second-order stationary stochastic process 
in a probability space $\mathbf L_2 (\Omega, \mathcal B,\mathcal P)$,
with correlation function $r(n-m)$.
Then the series
\[
X(z) = \sum_{n=0}^{\infty} z^n x_n, \qquad |z|<1
\]
converges in $\mathbf L_2 (\Omega, \mathcal B,\mathcal P)$ and we have
\[
E[X(z) \overline{X(w)}] = \frac{\phi(z) + \phi(w)^*}{2(1-z \overline{w})},
\]
where $\phi(z) = r(0) + 2\sum_{n=1}^\infty {r(n) z^n}$.

\subsection{Stationary processes over the white noise space}
\label{rk1234}
In our setting, we consider a sequence
$\left( x_n \right)_{n=0}^\infty$ in $\mathscr {S}_{-1}$ and
assume that the sum
$X(z) = \sum_{n=0}^{\infty} z^n x_n$
converges in $\mathscr {S}_{-1}$.
Then, by Remark \ref{rkKond}, there exists $\ell>0$
such that the sum $X(z)$ converges in $\mathcal H_\ell$.\smallskip

Hence the following definition makes sense.
\begin{definition}
\label{stat} 
Let $(x_n)_{n \in \mathbb N}$ be a sequence of
elements in $\mathcal{H}_\ell$ for some fixed $\ell \in \mathbb N$. The
stochastic sequence $(x_n)_{n \in \mathbb N}$ is called
\index{Wide-sense stationarity}wide-sense stationary if
\[
M^*_{x_n}M_{x_m} = R_x(m-n),
\]
depends only on $m-n$.
\end{definition}

Here we also assumed that the $z$-transform of the impulse response coefficients $(h_n)_{n \in \mathbb Z}$ converges in $\mathscr {S}_{-1}$,
and therefore it is also convergent in $\mathcal H_\ell$ for some $\ell \in \N$.

\begin{definition}
Let $(x_n)_{n \in \mathbb N}$ and $(y_n)_{n \in \mathbb N}$
be two sequences of elements of $ \mathcal{H}_\ell$ for some $\ell \in
\mathbb N$. The stochastic sequences are called jointly
wide-sense stationary if the stochastic sequence $z_n = [x_n \
y_n]^T $ is a stationary process in the sense of Definition \ref{stat}. 
In such a case, the cross-correlation operator
\[
M^*_{y_n}M_{x_m} = R_{xy}(m-n)
\]
depends only on $m-n$.
\end{definition}
\begin{theorem}
\label{corr1}
Let $(x_n)_{n \in \mathbb Z}$, a sequence of elements in $\mathcal H _k$, 
be a wide-sense stationary
process in the white noise space such that
\[
\sum_{m \in \mathbb Z}{||M^*_{x_n}M_{x_{n+m}}||_{\mathcal H _k
}} < \infty, \quad \forall n \in \mathbb Z
\]
and let $(h_n)_{n \in \mathbb Z}$ be a sequence of elements in
${\mathcal H _\ell }$ such that
\begin{equation}
\label{finiteNormSum}
\sum_{n \in \mathbb Z}{||M_{h_n}||_{\mathcal H _\ell }} < \infty.
\end{equation}
Then, the output $(y_n)_{n \in \mathbb Z}$ sequence of the LTI stochastic system 
characterized by the impulse response $(h_n)_{n \in \mathbb Z}$, i.e.
\begin{equation}
\label{ioLti}
y_n=\sum_{m\in\mathbb Z}h_{n-m}\circ x_m,\quad n\in\mathbb Z,
\end{equation}
belongs to $\mathcal H _k$ and is stationary. 
\end{theorem}

To prove Theorem \ref{corr1}, we first recall F. Mertens' Theorem on convolutions of sequences.

\begin{theorem}[Mertens \cite{mertens}]
\label{abs_wiener}
Assume $( a_n )_{n \in \mathbb N}$ and $( b_n)_{n \in \mathbb N}$ are
two sequences of complex numbers such that
\[
\sum_{n=0}^{\infty}{|a_n|} < \infty \quad{and}\quad
\sum_{n=0}^{\infty}{|b_n|} < \infty,
\]
and let $( c_n )_{n \in \mathbb N}$ be defined by
$c_n=\sum_{m=0}^{n}{a_{n-m}b_m}$. Then
\[
\sum_{n=0}^{\infty}{|c_n|} < \infty
\]
and
\[
\sum_{n=0}^{\infty}{c_n}
=
\left( \sum_{n=0}^{\infty}{a_n}\right) \left(\sum_{n=0}^{\infty}{b_n}\right).
\]
\end{theorem}
\begin{pf}[of Theorem \ref{corr1}]
We first note that
\begin{align*}
M_{y_{m}}^*M_{y_{n+m}}
& =
M_{\sum_{t \in \mathbb Z}{x_{m-t}\circ h_t}}^*M_{\sum_{s \in \mathbb Z}{x_{n+m-s}\circ h_s}}
\\ & =
\sum_{t,s=1}^{\infty}{M_{h_t}^*M_{x_{m-t}}^*M_{x_{n+m-s}}M_{h_{s}}},
\end{align*}
therefore
\[
M_{y_{m}}^*M_{y_{n+m}}
=
\sum_{t,k=1}^{\infty}{M_{h_t}^*R_x(k)M_{h_{n+t-k}}}
\]
depends only on $n$ and, by definition, the stationarity of the output signal follows.
Furthermore, we have
\begin{align}
\nonumber 
\sum_{t,s= -\infty}^{\infty}&{||M_{h_t}^* M_{x_{m-t}}^*M_{x_{n+m-s}} M_{h_{s}}||_{\mathcal{H}_k}}
\leq
\\
&
\leq
\sum_{t,s=-\infty}^{\infty}{||M_{h_t}^*||_{\mathcal{H}_k} \ ||M_{x_{m-t}}^*M_{x_{n+m-s}}||_{\mathcal{H}_k} \ ||M_{h_{s}}||_{\mathcal{H}_k}}.
\label{y_conv1}
\end{align}
Then, using Theorem \ref{abs_wiener} twice, the expression
\begin{equation}
\label{y_conv}
\sum_{t,s=1}^{\infty}{M_{h_t}^*M_{x_{m-t}}^*M_{x_{n+m-s}}M_{h_{s}}}
\end{equation}
is absolutely convergent, i.e.
\[
\sum_{t,s=1}^{\infty}{||M_{h_t}^*M_{x_{m-t}}^*M_{x_{n+m-s}}M_{h_{s}}||_{\mathcal{H}_k}} <\infty.
\]
\end{pf}
We note that, when the system parameters are assumed deterministic (in view of Lemma \ref{WickDet}), 
the definitions and the results are reduced to the classical case.
\subsection{The spectrum}
\label{SectionSpectrum}

In the study of stochastic processes, a specific case is of
interest, namely when the joint probability distribution does
not change while shifting in time (moving the indices by a
constant). 
In particular, the autocorrelation operator depends
only on the difference of the indices. This property underlines
the notion of wide sense stationary processes. 
In order to proceed, the counterparts of known definitions should be given within the white noise space framework. As in the classical case,
the spectrum is given in term of the Fourier series.
\begin{definition}
Let $\ell \in \mathbb N$ and let $(x_n)_{n \in \mathbb Z}$ be a stationary process with elements in $\mathcal{H}_\ell$.
Then the spectrum of a stationary process is the operator-valued function defined by
\[
S_x(e^{\imath  \omega})
=
\sum_{m=-\infty}^{\infty}{M^*_{x_n}M_{x_{n+m}} e ^ {\imath \omega m}}
=
\sum_{m=-\infty}^{\infty}{R_x(m) e ^ {\imath \omega m}}
\]
where we assume
\[
\sum_{m \in \mathbb Z}{||M^*_{x_n}M_{x_{n+m}}||_{\mathcal H _k
}} < \infty.
\]
\end{definition}
The above assumption is required in order to ensure the
expression $\sum_{m=-\infty}^{\infty}{R_x(m) e ^ {\imath \omega
m}}$ belongs to the operator-valued Wiener algebra.\\

The following result describes the relationship between the input spectrum and the
output spectrum through a linear time invariant (LTI) system in
the white noise space setting \eqref{ioLti}.
\begin{theorem}
\label{corr2}
Let $(x_n)_{n \in \mathbb Z}$ be a wide-sense stationary
processes in $\mathscr {S}_{-1}$, such that
\[
\sum_{m \in \mathbb Z}{||M^*_{x_n}M_{x_{n+m}}||_{\mathcal H _k
}} < \infty, \quad \forall n \in \mathbb Z
\]
and let $(h_n)_{n \in \mathbb Z}$ be a sequence of elements in
${\mathcal H _\ell }$ such that Condition \eqref{finiteNormSum} holds.
Then the spectrum of $(y_n)_{n \in \mathbb Z}$ is well-defined and is given by
\[
S_y(e^{\imath  \omega})=\mathscr Z(M_h^*)(e^{-\imath  \omega})S_x(e^{\imath  \omega})\mathscr Z(M_h)(e^{\imath  \omega}),
\]
where $\mathscr Z$ denotes the element in the Wiener algebra associated to a sequence of operators, or equivalently,
\[
\mathscr Z(M_{h_n})(e^{\imath  \omega})
\defEq
\sum _{n \in \mathbb Z} {M_{h_n} e^{\imath  \omega n}  }.
\]
\end{theorem}
\begin{pf}
Using Theorem \ref{corr1}, $S_y(e^{\imath  \omega})$ is well defined and belongs to
the associated Wiener algebra (over $\mathcal H _ k$).
A direct calculation, using \eqref{y_conv1}, completes the proof:
\begin{eqnarray*}
S_y(e^{\imath  \omega})
& = &
\sum_{n\in \mathbb Z}{M_{y_{m}}^*M_{y_{n+m}}e^{\imath \omega n}} 
\\
& = & 
\sum_{n \in \mathbb Z}{\left( \sum_{t,s\in \mathbb Z}{M_{h_t}^*M_{x_{m-t}}^*M_{x_{n+m-s}} M_{h_{s}}}\right) e^{\imath \omega n}} 
\\
& = &
\sum_{n \in \mathbb Z}{(\sum_{t,s \in \mathbb Z}{M_{h_t}^*R_x(n+t-s)M_{h_{s}}})e^{\imath\omega n}} 
\\
& = &
\sum_{t,s \in \mathbb Z}{\sum_{n \in \mathbb Z}{M_{h_t}^*R_x(n+t-s)M_{h_{s}}}e^{\imath\omega (n+t-t+s-s)}} 
\\
& = &
\sum_{t,s \in \mathbb Z}{\sum_{u \in \mathbb Z}{M_{h_t}^*R_x(u)M_{h_{s}}}e^{\imath\omega (u-t+s)}} 
\\
& = &
\left(\sum_{t \in \mathbb Z}{M_{h_t}^*e^{-\imath\omega t}}\right)
\left(\sum_{u \in \mathbb Z}{R_x(u)e^{\imath\omega u}}\right)
\left(\sum_{s \in \mathbb Z}{M_{h_s}e^{\imath\omega s}}\right)\\
& = &
\mathscr{Z}(M_h^*)(e^{ - \imath \omega})S_x(e^{\imath \omega}) \mathscr{Z}(M_h)(e^{\imath \omega}).
\end{eqnarray*}
\end{pf}

\section{Operator-valued stochastic Wiener Filter}
\label{SectionWienerFilter}
\setcounter{equation}{0}

In this section, we follow the
classical approach of the Wiener filter construction, see
\cite[Chapter 7]{Hassibi96T}, and adapt it to the white noise space
framework.\smallskip

There are two equivalent fundamental starting points which both lead
to the classical Wiener filter. The first is by derivation of the
mean square error and the second is based on the orthogonality
between the error and the filter input. We develop the white
noise space approach to the Wiener filter based on the
orthogonality property.
The main result is presented below in Theorem \ref{thmLinearOpWie}
and the stochastic filter is in \eqref{noncausalWienerFilterWns}.
\smallskip

The orthogonality in the setting of the white noise space is defined
as follows:
\begin{definition}
Let $\ell,k \in \mathbb{N}$ be such that $k > \ell +1$. 
The elements $x,y\in\mathcal H_\ell$ are orthogonal in the white
noise space sense if
\[
\innerProductTri{M_{x} \eta }{ M_{y} \xi }{\mathcal{H}_k} =0, \quad \forall \eta ,
\xi \in \mathcal{H}_k,
\]that is,
\[
\innerProductTri{x \circ \eta }{ y \circ \xi }{\mathcal{H}_k}=0, \quad \forall
\eta , \xi \in \mathcal{H}_k.
\]
\end{definition}
This definition makes it possible to give a geometric 
interpretation of multiplication operators via the geometry $\mathcal{H}_k$.
\smallskip

We consider two Hilbert space valued, discrete time, stationary and jointly
stationary random sequences $(u_i)_{i \in \mathbb N}$ and
$(y_i)_{i \in \mathbb N}$ belong $\mathcal H_k$.
Then the estimated process, by assumption, is an output of an LTI system, for every $j$, and hence is of the form
\begin{equation}
\label{estimator} 
\widehat{u_j} 
\defEq
\sum_{m=-\infty}^{\infty}{y_m \circ K_{j,m}} \quad j \in
\mathbb Z.
\end{equation}
The main goal is to describe $(K_{j,m}) \in \mathcal{H}_\ell $ such
that error is minimal. 
The orthogonality of the error and the output is
\[
(\widehat{u_j}-u_j) \quad
\perp \quad y_{l} \quad \forall l , j \in \mathbb Z,
\]
or, equivalently (by \ref{estimator}),
\[
 (\sum_{m=-\infty}^{\infty}{y_m \circ K_{j,m}}-u_j)
\quad \perp
\quad y_{l}
\quad \forall l \in \mathbb Z .
\]
By definition we have
\[
\innerProductTri{(M_{\widehat{u_j}}-M_{u_j}) \eta }{ M_{y_l} \xi }{\mathcal{H}_k}
=0 \quad \forall \eta , \xi \in \mathcal{H}_k
\quad {\rm and} \quad \forall l \in \mathbb Z,
\]
which yields the following
\begin{equation}
\label{wienEq8}
\innerProductTri{M_{\widehat{u_j}}\eta }{ M_{y_l} \xi }{\mathcal{H}_k}
=
\innerProductTri{M_{u_j} \eta }{ M_{y_l} \xi }{\mathcal{H}_k}
\quad
\forall \eta , \xi \in \mathcal{H}_k
\quad {\rm and }
\quad \forall l \in \mathbb Z.
\end{equation}
Then, by using \eqref{wienEq8} and \eqref{estimator}, we have
\begin{eqnarray*}
M^*_{y _l}M_{u _ j}
& = & M^*_{y _l}M_{\widehat{u_j}} \\
& = & M^*_{y _l}M_{\sum_{m=-\infty}^{\infty}{y_m \circ K_{j,m}}} \\
& = & \sum_{m=-\infty}^{\infty} { M^*_{y_l}M_{y_m} M_{K_{j,m}}  } \\
& = & \sum_{m=-\infty}^{\infty}{R_{y}(m - l)M_{K_{j,m}}}.
\end{eqnarray*}
This implies a autocorrelation operator-valued equality,
also known as the Wiener-Hopf equations, which are given by
\begin{equation}
\label{form1} R_{uy}(j - l)= \sum_{m=-\infty}^{\infty}
{R_{y}(m - l) M_{K_{j,m}} } \quad \forall l \in
\mathbb Z.
\end{equation}
We have an infinite set of linear equations indexed by $l$ and we show, by stationarity, that this set reduces to an equation
solvable by the Fourier series representation, namely, its spectrum. 
We first simplify the indices of the filter $K$. 
By the following change of variables $m-l=m'$ and $j-l=j'$, we have
\begin{equation}
\label{Auto} R_{uy}(j')
=\sum_{m'=-\infty}^{\infty}{R_{y}(m')M_{K_{j'+l,m'+l}
}} \quad \forall l \in \mathbb Z.
\end{equation}
The left-hand side of \eqref{Auto}
is independent of $l$.
This implies there exists a sequence 
$( \widehat{K}_{j - l} )_{j - l \in \mathbb N}$ such that
\[ 
M_{K_{j ' + l , m' +l}}
= 
M_{K_{j '  , m' }} = M_{\widehat{K}_{j ' - m '}}
.
\]
Hence \eqref{form1} is reduced to
\[
R_{uy}(j)= \sum_{m=-\infty}^{\infty}{R_{y}(m)
M_{\widehat{K}_{j - l}} },
\]
and, as a consequence, \eqref{estimator} may be be rewritten as
\[
\widehat{u}_j= \sum_{m=-\infty}^{\infty}{y_m  \circ \widehat{K}_{j - m} }.
\]
Note that this expression implicitly contains double
convolution, one is obvious and the second is due to the Wick
product. By applying the Fourier series to the autocorrelation
function, and noticing that convolution becomes point--wise
multiplication in the Wiener algebra, we have
\[
S_{uy}(e^{ \imath \omega}) = S_{y}(e^{ \imath \omega})
{\widehat{K}} (e^{ \imath \omega}) .
\]
We wish to invert $S_{y}(e^{ \imath \omega})$. To that end,
we assume that
\begin{equation}
\label {WA}
\sum_{n=-\infty}^{\infty}{||R_{x}(n)||_{\mathcal H _k}}<\infty 
\quad {\rm and} \quad
\sum_{n=-\infty}^{\infty}{||R_{y}(n)||_{\mathcal H _k}}<\infty,
\end{equation}
where $|| \cdot ||_{\mathcal H _k}$ denotes the operator norm in ${\mathcal H _k}$. 
Under these assumptions, $S_{y}$ belongs to the associated Wiener algebra. 
If, furthermore, we assume that $S_{y}(e^{
\imath \omega})$ is positive definite for every
$\omega \in \mathbb R$, by the operator-valued factorization theorem, Theorem
\ref{Spectral}, $S_{y}^{-1}$ exists in the associated Wiener algebra.
In this case, we conclude that the Wiener algebra element associated with the operators of the linear filter is given by:
\begin{equation}
{\widehat{K}} (e^{ \imath \omega}) 
=
 S_{y}^{-1}(e^{ \imath \omega})   S_{uy}(e^{ \imath \omega}) .
\end{equation}
We summarize the above in the following theorem.
\begin{theorem}[Linear optimal discrete-time filter - The WNS version]
\label{thmLinearOpWie}
Fix $k,\ell \geq 0$ such that $k > \ell + 1$ and assume that conditions \eqref{WA} are in force. Let $(u_n)_{n
\in \mathbb Z}$ and $(y_n)_{n \in \mathbb Z}$ be two stationary
processes over $\mathcal H _ \ell$ such that they are jointly stationary. 
Then the linear filter \eqref{estimator} is given by
\begin{equation}
\label{noncausalWienerFilterWns}
{\widehat{K}} (e^{ \imath \omega}) 
=
S_{y}^{-1}(e^{ \imath \omega})   S_{uy}(e^{ \imath \omega})
\end{equation}
and belongs to $\mathcal W_{\mathcal B(\mathcal H _k)}$.
\end{theorem}

\begin{Rk}[{Reduction to the Classical case}]
\label{rkdetermCase}

Let us assume that the filter $\left(K_{j,m}\right)_{j,m \in \N}$ in \eqref{estimator} is a sequence of complex number or, 
equivalently, represents the impulse response of a non-random system.
Hence, see \eqref{WickDet}, the Wick product reduces to the pointwise multiplication operator.
Furthermore, assuming that $\left(y_m\right)_{m \in \N}$ and $\left(u_m\right)_{m \in \N}$ 
are sequences which belong to ${\mathbf L}_2(\mathcal S^{\prime}, {\mathcal F},dP)$,
the operators $M_{x_m}^*M_{x_{m+n}}$ (see the equality in \eqref{extpectMult}) reduce now to scalar operators, 
i.e. $\lambda_n \, I$, where $\lambda_n \in \mathbb C$.\smallskip

Hence, $R_y(n)$ and $R_{uy}(n)$ are now scalar-valued functions instead of operator-valued functions.
Consequently, $S_{y}$ and $S_{uy}$ are members of the 
Wiener algebra with complex coefficients.
As a consequence, we conclude that the suggested stochastic model is in fact a generalization of the classical Wiener filter.
The result for the causal Wiener filter in the next section follows similarly.

\end{Rk}

\section{Stochastic causal Wiener filter}
\label{SectionCausalWienerFilter}
\setcounter{equation}{0}

A problem arises when one restricts  interest to causal filters.
In such cases, following the notations and the analysis in \cite{Hassibi96T}, 
denoting the estimated signal by $\widehat{u}_{j | j}$, we have
\[
\widehat{u}_{j} = \sum_{m=-\infty}^{ j}{y_m \circ K_{j,m}}.
\]
The orthogonality principle becomes
\[
(\widehat{u}_{j}-u_j) \quad \perp
\quad y_{l} \quad \forall l .
\]
The double indices of $K_{j,m}$ may be reduced to a single index $K_{m}$.
Using the same method used in the previous section we can rewrite
the Wiener-Hopf equations \eqref{form1} as
\[
R_{uy}(j)= \sum_{m=-\infty}^{j}{ R_{y}(m)  M_{K_{j
- m}}} =\sum_{m=0}^{\infty}{R_{y}(j - m)  M_{K_{m}}} \quad
\forall j \geq 0,
\]
or, equivalently, as
\begin{equation}
\label{WH}
R_{uy}(j)= \sum_{m=-\infty}^{\infty}{ R_{y}(m)
M_{K_{j - m} }} =\sum_{m=-\infty}^{\infty}{ R_{y}(j - m)
M_{K_{m} }} \quad \forall j \geq 0,
\end{equation}
where
\[
K_m=0 \quad \forall m<0.
\]
Equation (\ref{WH}) is difficult to solve because the equality is
valid only for $j \geq 0$. 
As a result, the Fourier series is not defined. 
\smallskip

In order to overcome the restriction in \eqref{WH} that $j$ is positive, we define the sequence 
$(g_j)_{j \in \mathbb Z}$ by:
\begin{equation}
\label{causal_eq} g_j \overset{ \text{def} }{=}
R_{uy}(j)  - \sum_{m=0}^{\infty}{ R_{y}(j - m)
M_{K_{m}}} \quad -\infty < j < \infty,
\end{equation}
which, by \eqref{WH}, $g_m=0$ for all $m \geq 0$.
Since $g_m$ is defined for all
$-\infty < m < \infty$,
we now may consider the Fourier series of equation \eqref{causal_eq} to obtain:
\begin{equation}
\label{causalFilEq1}
G(e^{\imath  \omega})=S_{uy}(e^{\imath  \omega})- S_{y}(e^{\imath  \omega}) K(e^{\imath  \omega}) .
\end{equation}
If $S_{y}(e^{\imath  \omega}) > 0$,
then by the operator-valued spectral factorization, Theorem \ref{Spectral}, we have the following factorization:
\begin{equation}
\label{causalFilEq2}
S_{y}(e^{\imath  \omega})=  W_S(e^{\imath  \omega})^* W_S(e^{\imath  \omega}).
\end{equation}
Substituting \eqref{causalFilEq2} in \eqref{causalFilEq1} leads to
\begin{equation}
\label{causalEq1}
W_S(e^{\imath  \omega})^{-*} G(e^{\imath  \omega}) = 
W_S(e^{\imath  \omega})^{-*} S_{uy}(e^{\imath  \omega}) - W_S(e^{\imath  \omega}) K(e^{\imath  \omega}).
\end{equation}
We note that, since $g_j$ is strictly anticausal, 
$G(e^{\imath  \omega}) \in  \mathcal{W}_- (\mathcal{B})$ 
and, similarly, since $K_m=0$ for all $m<0$,
we have 
$K(e^{i \omega}) \in  \mathcal{W}_+ (\mathcal{B})$.
Furthermore, by \eqref{Spectral}, 
$W_S(e^{\imath \omega}) \in  \mathcal{W}_+ (\mathcal{B})$ and 
$W_S(e^{\imath \omega})^{-*} \in  \mathcal{W}_- (\mathcal{B})$. 
We therefore conclude that we have
\begin{equation}
\label{causalEq11}
W_S(e^{\imath  \omega}) K(e^{\imath  \omega})    		  	\in  \mathcal{W}_+ (\mathcal{B}) 
\end{equation}
and 
\begin{equation}
\label{causalEq12}
W_S(e^{\imath  \omega})^{-*} G(e^{\imath  \omega})  \in  \mathcal{W}_- (\mathcal{B}).
\end{equation}
We denote by $\mathcal C \left( W \right)$ the causal part
of an element $W$ in the corresponding Wiener algebra. 
Taking the causal part of Equation \ref{causalEq1}, 
one has, using \eqref{causalEq11} and \eqref{causalEq12}, the following
\begin{align}
\nonumber
\mathcal C & \left(W_S(e^{\imath  \omega})^{-*}  G(e^{\imath  \omega})\right)
=
\\ &
\mathcal C \left( W_S(e^{\imath  \omega})^{-*} S_{uz}(e^{\imath  \omega}) \right) - 
\mathcal C \left( W_S(e^{\imath  \omega}) K(e^{\imath  \omega}) \right).
\label{causalEq2}
\end{align}
By the spectral factorization, we have
$W_S(e^{\imath  \omega})\in \mathcal{W}_+ (\mathcal{B}) $
and 
$W_S(e^{\imath  \omega})^{-1} \in \mathcal{W}_+ (\mathcal{B})$. 
It follows that the
left hand side of \eqref{causalEq2} is strictly anticausal 
while the term on the right hand side is causal, and so
\begin{equation}
\label{causalEq3}
\mathcal C \left( W_S(e^{\imath  \omega})^{-*} G(e^{\imath  \omega}) \right) = 0    
\end{equation}
and
\begin{equation}
\label{causalEq4}
\mathcal C \left(  W_S(e^{\imath  \omega}) K(e^{\imath  \omega})  \right) =  W_S(e^{\imath  \omega})K(e^{\imath  \omega}).
\end{equation}
Substituting \eqref{causalEq3} and \eqref{causalEq4} into \eqref{causalEq2} and multiplying by $W_S(e^{\imath  \omega})^{-1}$ on the left, leads to
\begin{equation}
\label{causal_Wiener}
K(e^{\imath  \omega})= W_S(e^{\imath  \omega})^{-1}  
\mathcal C 
\left(
W_S(e^{\imath  \omega})^{-*}  S_{uz}(e^{\imath  \omega}) 
\right),
\end{equation}

where the components in \eqref{causal_Wiener} are
operator-valued functions.

\begin{Ex}[The case of additive noise]
{\rm 
Let us consider the question of filtering $(x_n)_{n \in \mathbb N}$
from $(y_n)_{n \in \mathbb N}$, where $(y_n)_{n \in \mathbb N}$ is
given by:
\[
y_n = x_n + v_n.
\]
Here $(v_n)_{n \in \mathbb N}$, the system noise,
has spectrum $S_v(e^{\imath  \omega})=\mathcal V _0$ 
where $\mathcal V$ is a constant positive operator.
Equivalently, we have
\[
M^*_{v_n}M_{v_m} 
\defEq 
R_v(m-n) 
= 
\mathcal V _0 \delta(m-n).
\]
We also assume that 
$\innerProductTri{v_n \circ \xi}{x_m \circ \eta}{\mathcal H _k} = 0$ for all $\xi,\eta \in \mathcal H _k$, 
i.e. $R_{xv}(n) = 0$ and so $S_{xv}(e^{\imath \omega})=0$.
As consequences, we may conclude the following
\[
S_{xv}(e^{\imath \omega}) = 0,
\qquad 
S_{y}(e^{\imath \omega}) = S_{x}(e^{\imath \omega}) + \mathcal V _0,
\]
and
\[
S_{xy}(e^{\imath \omega}) = S_{x}(e^{\imath \omega}).
\]
Hence, in the additive noise case, the non-causal Wiener filter \eqref{noncausalWienerFilterWns} becomes
\[
K(e^{\imath  \omega})
=
S_y(e^{\imath  \omega})^{-1} S_{uy}(e^{\imath  \omega})
=
(S_{x}(e^{\imath \omega}) + \mathcal V _0)^{-1} S_{uy}(e^{\imath  \omega}).
\]
The causal Wiener filter, 
developed above in \eqref{causal_Wiener}, is now given by:
\begin{align*}
K(e^{\imath  \omega}) 
= &
W_S(e^{\imath \omega})^{-1}
\mathcal C 
\left(
W_S(e^{\imath \omega})^{-*}
S_{xy}(e^{\imath \omega})
\right)
\\ = &
W_S(e^{\imath \omega})^{-1}
\mathcal C 
\left(
W_S(e^{\imath \omega})^{-*}
(S_{y}(e^{\imath \omega}) - \mathcal V_0)
\right)
\\
= &
W_S(e^{\imath \omega})^{-1}
\mathcal C 
\left(
W_S(e^{\imath \omega})
-
W_S(e^{\imath \omega})^{-*} 
\mathcal V _0 
\right)
.
\end{align*}
Finally, since
\[
\mathcal C \left( W_S(e^{\imath \omega})\right)
=
W_S(e^{\imath \omega})
\quad {\rm and} \quad
\mathcal C 
\left( 
W_S(e^{\imath \omega})^{-*} \mathcal V _0 
\right) 
= \mathcal V _0,
\]
we may conclude that
\[
K(e^{\imath \omega}) 
= 
I_{\mathcal H_k} 
- 
W_S(e^{\imath \omega})^{-1}  \mathcal V _0.
\]
} 
\end{Ex}
\def\cfgrv#1{\ifmmode\setbox7\hbox{$\accent"5E#1$}\else
  \setbox7\hbox{\accent"5E#1}\penalty 10000\relax\fi\raise 1\ht7
  \hbox{\lower1.05ex\hbox to 1\wd7{\hss\accent"12\hss}}\penalty 10000
  \hskip-1\wd7\penalty 10000\box7} \def\cprime{$'$} \def\cprime{$'$}
  \def\cprime{$'$} \def\lfhook#1{\setbox0=\hbox{#1}{\ooalign{\hidewidth
  \lower1.5ex\hbox{'}\hidewidth\crcr\unhbox0}}} \def\cprime{$'$}
  \def\cprime{$'$} \def\cprime{$'$} \def\cprime{$'$} \def\cprime{$'$}
  \def\cprime{$'$}

\end{document}